\documentclass{amsart}
\usepackage{amsmath,amsthm,amssymb,mathrsfs}

\usepackage{amsrefs}             
                                 
\usepackage[all]{xy}
\usepackage{hyperref}
\newtheorem{thm}{Theorem}[section]

\newtheorem{lemma}[thm]{Lemma}
\newtheorem{prop}[thm]{Proposition}

\theoremstyle{definition}
\newtheorem{definition}[thm]{Definition}

\theoremstyle{remark}
\newtheorem{remark}[thm]{Remark}

\newtheorem*{claim}{Claim}
\def\mathcs{C^{*}}
\newcommand{\cs}{\ensuremath{\mathcs}}

\DeclareMathSymbol{\rtimes}{\mathbin}{AMSb}{"6F}

\def\K{\mathcal{K}}

\DeclareMathOperator*{\supp}{supp}

\DeclareMathOperator{\Aut}{Aut}

\DeclareMathOperator{\id}{id}
\def\set#1{\{\,#1\,\}}
\newcommand\sset[1]{\{#1\}}

\def\restr#1{|_{{#1}}}
\makeatletter
\def\labelenumi{\textnormal{(\@alph\c@enumi)}}
\def\theenumi{\@alph \c@enumi}
\def\labelenumii{\textnormal{(\@roman\c@enumii)}}
\def\theenumii{\@roman \c@enumii}
\newcount\charno
\def\alphapart#1{\charno=96
\advance\charno by#1\char\charno}

\makeatother

\def\<{\langle}
\def\>{\rangle}
\let\ipscriptstyle=\scriptscriptstyle
\def\lipsqueeze{{\mskip -3.0mu}}
\def\ripsqueeze{{\mskip -3.0mu}}
\def\ipcomma{\nobreak\mathrel{,}\nobreak}
\newbox\ipstrutbox
\setbox\ipstrutbox=\hbox{\vrule height8.5pt
width 0pt}
\def\ipstrut{\copy\ipstrutbox}
\def\lip#1<#2,#3>{\mathopen{\relax_{\ipstrut\ipscriptstyle{
#1}}\lipsqueeze
\langle} #2\ipcomma #3 \rangle}
\def\blip#1<#2,#3>{\mathopen{\relax_{\ipstrut
\ipscriptstyle{ #1}}\lipsqueeze\bigl\langle} #2\ipcomma #3 \bigr\rangle}
\def\rip#1<#2,#3>{\langle #2\ipcomma #3
\rangle_{\ripsqueeze\ipstrut\ipscriptstyle{#1}}}
\def\brip#1<#2,#3>{\bigl\langle #2\ipcomma #3
\bigr\rangle_{\ripsqueeze\ipstrut\ipscriptstyle{#1}}}
\def\angsqueeze{\mskip -6mu}
\def\smangsqueeze{\mskip -3.7mu}
\def\trip#1<#2,#3>{\langle\smangsqueeze\langle #2\ipcomma #3
\rangle\smangsqueeze\rangle_{\ripsqueeze\ipstrut\ipscriptstyle{#1}}}
\def\btrip#1<#2,#3>{\bigl\langle\angsqueeze\bigl\langle #2\ipcomma
#3
\bigr\rangle
\angsqueeze\bigr\rangle_{\ripsqueeze\ipstrut\ipscriptstyle{#1}}}
\def\tlip#1<#2,#3>{\mathopen{\relax_{\ipstrut\ipscriptstyle{
#1}}\lipsqueeze \langle\smangsqueeze\langle} #2\ipcomma #3
\rangle\smangsqueeze\rangle}
\def\btlip#1<#2,#3>{\mathopen{\relax_{\ipstrut\ipscriptstyle{
#1}}\lipsqueeze
\bigl\langle\angsqueeze\bigl\langle} #2\ipcomma #3
\bigr\rangle\angsqueeze\bigr\rangle}

\def\ip(#1|#2){(#1\mid #2)}
\def\bip(#1|#2){\bigl(#1 \mid #2\bigr)}
\def\Bip(#1|#2){\Bigl( #1 \bigm| #2 \Bigr)}
\expandafter\ifx\csname BibSpec\endcsname\relax\else
\BibSpec{collection.article}{%
    +{}  {\PrintAuthors}                {author}
    +{,} { \textit}                     {title}
    +{.} { }                            {part}
    +{:} { \textit}                     {subtitle}
    +{,} { \PrintContributions}         {contribution}
    +{,} { \PrintConference}            {conference}
    +{}  {\PrintBook}                   {book}
    +{,} { }                            {booktitle}
    +{,} { }                           {series}
    +{,} { \voltext}                    {volume}
    +{,} { }                            {publisher}
    +{,} { }                            {organization}
    +{,} { }                            {address}
    +{,} { \PrintDateB}                 {date}
    +{,} { pp.~}                        {pages}
    +{,} { }                            {status}
    +{,} { \PrintDOI}                   {doi}
    +{,} { available at \eprint}        {eprint}
    +{}  { \parenthesize}               {language}
    +{}  { \PrintTranslation}           {translation}
    +{;} { \PrintReprint}               {reprint}
    +{.} { }                            {note}
    +{.} {}                             {transition}
}
\BibSpec{article}{%
    +{}  {\PrintAuthors}                {author}
    +{,} { \textit}                     {title}
    +{.} { }                            {part}
    +{:} { \textit}                     {subtitle}
    +{,} { \PrintContributions}         {contribution}
    +{.} { \PrintPartials}              {partial}
    +{,} { }                            {journal}
    +{}  { \textbf}                     {volume}
    +{}  { \PrintDatePV}                {date}
    +{,} { \eprintpages}                {pages}
    +{,} { }                            {status}
    +{,} { \PrintDOI}                   {doi}
    +{,} { available at \eprint}        {eprint}
    +{}  { \parenthesize}               {language}
    +{}  { \PrintTranslation}           {translation}
    +{;} { \PrintReprint}               {reprint}
    +{.} { }                            {note}
    +{.} {}                             {transition}
}
\BibSpec{book}{%
    +{}  {\PrintPrimary}                {transition}
    +{,} { \textit}                     {title}
    +{.} { }                            {part}
    +{:} { \textit}                     {subtitle}
    +{,} { \PrintEdition}               {edition}
    +{}  { \PrintEditorsB}              {editor}
    +{,} { \PrintTranslatorsC}          {translator}
    +{,} { \PrintContributions}         {contribution}
    +{,} { }                            {series}
    +{,} { \voltext}                    {volume}
    +{,} { }                            {publisher}
    +{,} { }                            {organization}
    +{,} { }                            {address}
    +{,} { pp.~}                        {pages}
    +{,} { \PrintDateB}                 {date}
    +{,} { }                            {status}
    +{}  { \parenthesize}               {language}
    +{}  { \PrintTranslation}           {translation}
    +{;} { \PrintReprint}               {reprint}
    +{.} { }                            {note}
    +{.} {}                             {transition}
}
\fi
\newcommand\groupfont{\mathcal}
\newcommand\bundlefont{\mathscr}

\renewcommand\lg{\lambda}
\newcommand\lgb{\lambda_{\Gu}}
\newcommand\lh{\sigma}
\newcommand\Gu{\underline{G}}
\newcommand\Hu{\underline{H}}
\newcommand\Eu{\underline{E}}
\newcommand\go{G^{(0)}}
\newcommand\ho{H^{(0)}}
\newcommand\eo{\Eu^{(0)}}
\newcommand\Ggr{\groupfont{G}}
\newcommand\Hgr{\groupfont{H}}

\newcommand\betab{\bar\beta}
\newcommand\alphab{\bar\alpha}

\newcommand{\inv}{^{-1}}
\DeclareMathOperator{\Iso}{Iso}
\DeclareMathOperator{\Rep}{Rep}
\renewcommand\L{\mathcal{L}}
\newcommand\unit{^{(0)}}
\newcommand\half{\frac12}
\def\neghalf(#1){(#1)^{-\half}}
\def\poshalf(#1){(#1)^{\half}}
\newcommand\A{\bundlefont{A}}
\newcommand\B{\bundlefont{B}}
\newcommand\CC{\bundlefont{C}}
\def\sa_#1(#2,#3){\Gamma_{#1}(#2;#3)}
\newcommand\rt{\operatorname{rt}}
\renewcommand\H{\mathcal{H}}
\newcommand\HH{\bundlefont{H}}
\newcommand\KK{\bundlefont{K}}

\allowdisplaybreaks
\usepackage{color}

\begin{document}
\title[Iterated Groupoid Crossed Products]{Groupoid Equivalence and
  the Associated Iterated Crossed Product} 

\author[J. H. Brown]{Jonathan Henry Brown}
\address{Department of Mathematics and Statistics, University of
  Otago, P.O. Box 56, Dunedin 9054 New Zealand}
\email{jonathan.henry.brown@gmail.com}

\author[G. Goehle]{Geoff Goehle}
\address{Mathematics and Computer Science Department, Stillwell 426,
  Western Carolina University, Cullowhee, NC 28723}
\email{grgoehle@email.wcu.edu}

\author[D. P. Williams]{Dana P. Williams}
\address{Department of Mathematics, 6188 Kemeny Hall, Dartmouth
  College, Hanover, NH 03755}
\email{dana.williams@dartmouth.edu}

\subjclass[2000]{46L55, 47L65, 22A22}
\keywords{groupoids, crossed products, equivalence theorem}
\begin{abstract}
  Given groupoids $G$ and $H$ and a $(G,H)$-equivalence $X$ we may form
  the transformation groupoid $G\ltimes X\rtimes H$.  Given a
  separable groupoid dynamical system $(A,G\ltimes X\rtimes H,\omega)$
  we may restrict $\omega$ to an action of $G\ltimes X$ on $A$ and
  form the crossed product $A\rtimes G\ltimes X$. We show that there
  is an action of $H$ on $A\rtimes G\ltimes X$ and that the iterated
  crossed product $(A\rtimes G\ltimes X)\rtimes H$ is naturally
  isomorphic to the crossed product $A\rtimes (G\ltimes X\rtimes H)$.
\end{abstract}

\date{10 June 2012}
\maketitle


\section{Introduction}
\label{sec:introduction}

If $\alpha:\Ggr\to\Aut A$ and $\beta:\Hgr\to\Aut A$ are commuting
actions of locally compact groups $\Ggr$ and $\Hgr$ on a \cs-algebra
$A$, then we trivially obtain an action
$\alpha\times\beta:\Ggr\times\Hgr\to\Aut A$.  Furthermore, every
action of $\Ggr\times\Hgr$ arises in this way.  It is straightforward
to check that the
crossed product $A\rtimes_{\alpha\times\beta}(\Ggr\times\Hgr)$
decomposes (up to isomorphism) as
$(A\rtimes_{\alpha}\Ggr)\rtimes_{\betab}H$, where
$\betab:=\beta\rtimes 1$ is the associated action of $\Hgr$ on
$A\rtimes_{\alpha}\Ggr$.

Recently, we discovered we needed a version of this iterated crossed
product result for groupoid dynamical systems.  However, it is far
from clear just what form such a general result would take.  For
example, groupoids act on fibred objects, so $A$ would need to be
fibred over the unit spaces of both groupoids.  Even then, it is not
so obvious what it should mean for the actions to commute.

Rather than sort out the most general possible theorem, we opted to
let our applications dictate the form of our result.
 In particular, we want to consider locally compact
Hausdorff groupoids $G$ and $H$ that are equivalent via a
$(G,H)$-equivalence $X$.  Then we can form the (bi)transformation
groupoid $G\ltimes X\rtimes H$ which naturally contains the
transformation groupoids $G\ltimes X$ and $X\rtimes H$ as
subgroupoids.  (To see how this set-up relates to commuting actions of
groups, see Remark~\ref{rem-group-case}.)  Given a groupoid dynamical
system $(A,G\ltimes X\rtimes H,\omega)$, we can get actions $\alpha$
of $G\ltimes X$ and $\beta$ of $X\rtimes H$ via restriction.  It turns
out that the crossed product $A\rtimes_{\alpha}(G\ltimes X)$ naturally
fibres over the orbit space $(G\ltimes X)\backslash X$.  Since the
latter equals $G\backslash X$ and since $G\backslash X$ is
homeomorphic to $\ho$ because $X$ is a $(G,H)$-equivalence, it is not
surprising that $\beta$ induces an action $\betab$ of $H$ on
$A\rtimes_{\alpha}(G\ltimes X)$.  Then our main result
(Theorem~\ref{thm:isomorphism}) states that the iterated crossed
product
\begin{equation*}
  \bigl(A\rtimes_{\alpha}(G\ltimes X)\bigr)\rtimes_{\betab}H
\end{equation*}
is naturally isomorphic to the crossed product $A\rtimes_\omega(G\ltimes
X\rtimes H)$.

Although such a result might be expected, especially in analogy with
the group case, it is nontrivial to prove.  Some of the
subtleties are foreshadowed by the proof of the ``standard
result'' in the group case.  Even in that setting, one needs to work
with covariant representations.  Of course covariant representations
of groupoid crossed products are considerably more subtle than their
classical counterparts, and come with an (un-)healthy dose of measure
theory.

Before we plunge into the details, we feel obligated to say a bit
about our main application.  In \cite{brogoe:xx12}, the first two authors
introduce the Brauer semigroup $S(G)$ of a locally compact groupoid
$G$.  The Brauer semigroup is a natural outgrowth of the Brauer group
$\operatorname{Br}(G)$ introduced in \cite{kmrw:ajm98} and consists of
Morita equivalence classes $[A,\alpha]$ of groupoid dynamical
$G$-systems with a suitable multiplication.  For example, if
$G=\Ggr\times X$ is the transformation groupoid associated to the
transformation \emph{group} $(\Ggr,X)$, then $S(G)$ is (isomorphic to)
the equivariant Brauer semigroup $S_{\Ggr}(X)$ of \cite{hrw:tams00}.
The main theorem in \cite{brogoe:xx12} is the full groupoid dynamical
system analogue of \cite{hrw:tams00}*{Theorem~5.2} and asserts that if
$G$ and $H$ are equivalent groupoids, then there is a semigroup
isomorphism $\Theta$ from $S(H)$ onto $S(G)$.  A critical additional
ingredient from  \cite{hrw:tams00} is that $\Theta$
is constructed such that if $\Theta([B,\beta])=[A,\alpha]$, then the
crossed products $B\rtimes_{\beta}H$ and $A\rtimes_{\alpha}G$ are
Morita equivalent. 

The proof in \cite{brogoe:xx12} is modeled on the proof in
\cite{hrw:tams00} and it goes as follows.  Suppose that $\mathfrak c$ is
a class in $S(H)$.  Then there is an associated dynamical system
$(A,G\ltimes X\rtimes H,\omega)$ such that the following statements
can be verified.  The action $\alpha=\omega\restr{G\ltimes X}$ is
proper and saturated as defined in \cite{bro:jot12}.  Consequently, we
can form the generalized fixed point algebra $A^{\alpha}$ which is
Morita equivalent to $A\rtimes_{\alpha}G$.  Furthermore, $A^{\alpha}$
is fibred over $G\backslash X\cong \ho$, and $\beta$ induces an action
$\hat \beta$ of $H$ on $A^{\alpha}$.  The point is that the
construction of $(A,G\ltimes X\rtimes H,\omega)$ is such that our class
$\mathfrak c$ must be of the form $[A^{\alpha},\hat \beta]$, and such
that there
is a well-defined map $\Theta$ sending $\mathfrak c$ to the class
$[A^{\beta},\hat\alpha]$ (with $A^{\beta}$ and $\hat\alpha$ defined
analogously to $A^{\alpha}$ and $\hat\beta$).  An application of the
Equivalence Theorem \cite{muhwil:nyjm08}*{Theorem~5.5} (as in
\cite{muhwil:nyjm08}*{\S9.1}) shows that
$A^{\alpha}\rtimes_{\hat\beta}H$ is Morita equivalent to
$(A\rtimes_{\alpha}(G\ltimes X))\rtimes_{\betab}H$.  Employing
symmetric arguments, we also have $A^{\beta}\rtimes_{\hat\alpha}G$
Morita equivalent to $(A\rtimes_{\beta}(X\rtimes
H))\rtimes_{\alphab}G$.  The main result from this paper (and a little
symmetry) implies that both $(A\rtimes_{\alpha}(G\ltimes
X))\rtimes_{\betab}H$ and $(A\rtimes_{\beta}(X\rtimes
H))\rtimes_{\alphab}G$ are isomorphic to
$A\rtimes_{\omega}(G\ltimes X\rtimes H)$.  This implies that
$A^{\alpha}\rtimes_{\hat\beta}H$ and $A^{\beta}\rtimes_{\hat\alpha}G$
are Morita equivalent, and provides the last bit of the main result in
\cite{brogoe:xx12}.

The structure of this paper is roughly as follows.  We start by
briefly reviewing the groupoid crossed product construction in Section
\ref{sec:group-cross-prod}.  In Section \ref{sec:iter-cross-prod} we
introduce the iterated product and build the outer action.
Section \ref{sec:isom-iter-prod} contains the main result, as
described above, although significant portions of the proof are
postponed until Section \ref{sec:proof-that-upsilon}.
As is usual, we assume that homomorphisms between \cs-algebras are
$*$-preserving.  Furthermore, representations of \cs-algebras are
assumed to be nondegenerate.  Since we use the Equivalence and
Disintegration Theorems, we require separability in a nontrivial way
(see Remark~\ref{rem-separable-basic}).
Hence virtually all the topological spaces that appear here are second
countable, and with the exception of $B(\H)$ and other multiplier
algebras, our \cs-algebras and Banach spaces are assumed to be
separable.  In particular, our groupoids are all second countable and
Hausdorff, and will be assumed to have a Haar system.


\section{Groupoid Crossed Products}
\label{sec:group-cross-prod}

For background on groupoid crossed products, we refer to the
exposition in  \cite{muhwil:nyjm08}.  For the correspondence between
$C_{0}(X)$-algebras and upper semicontinuous $C^{*}$-bundles we refer
to \cite{wil:crossed}*{Appendix~C}.
For convenience, we review some of the basics here.

An upper semicontinuous $C^*$-bundle $\A$ over $X$ is a continuous
open surjection $p:\A\to X$ such that $A(x):=p\inv(x)$ is a
$C^*$-algebra for each $x\in X$ and which satisfies some additional continuity
conditions \cite[Definition~C.16]{wil:crossed}. In particular, the
algebra $A:=\sa_{0}(X,\A)$ of continuous sections vanishing at infinity
is a $\cs$-algebra with respect to the supremum norm.  On the other
hand, we say that a \cs-algebra $A$ is a
$C_0(X)$-algebra if it comes equipped with  a nondegenerate homomorphism of
$C_{0}(X)$ into the center of the multiplier algebra $M(A)$.  There is
a correspondence between $C_{0}(X)$-algebras and upper semicontinuous
\cs-bundles over $X$: given a $C_{0}(X)$-algebra $A$, there is a
bundle $\A$ such that $A$ is $C_{0}(X)$-isomorphic to $\sa_{0}(X,\A)$
(with its natural $C_{0}(X)$-action) \cite{wil:crossed}*{Theorem~C.26}.

Let $G$ be a locally compact Hausdorff groupoid, with unit space
$G^{(0)}$, Haar system $\sset{\lambda^u}_{u\in\go}$, and range and
source maps $r_G$ and $s_G$ respectively \cite{ren:groupoid}.  We
omit the subscripts on $r_G$ and $ s_G$ when the domain is clear from
context. Let $A$ be a $C_0(G^{(0)})$-algebra and $p:\A\to G^{(0)}$ be
the associated upper semicontinuous $C^*$-bundle.  An action $\alpha$
of $G$ on $A$ is a set of isomorphisms
$\{\alpha_{\gamma}:A(s(\gamma))\to A(r(\gamma))\}_{\gamma\in G}$ such
that $\alpha_\gamma\circ\alpha_\eta=\alpha_{\gamma\eta}$ and the map
$(\gamma,a)\mapsto \alpha_{\gamma}(a)$ is jointly continuous
\cite[Definition~4.1]{muhwil:nyjm08}. We refer to the triple
$(\A,G,\alpha)$ as a groupoid dynamical system. Let
$r^*\A=\{(\gamma,a)\in G\times\A: r(\gamma)=p(a)\}$ be the pull back
bundle of $\A$.   Let $\sa_{c}(G, r^*\A)$ be the
continuous compactly supported sections of $r^*\A$.

\begin{prop}[{\cite[Proposition~4.4]{muhwil:nyjm08}}] The space of 
  compactly supported sections, $\sa_{c}(G,
  r^*\A)$,  is a $*$-algebra with respect to the operations
  \begin{equation*}
    f*g(\gamma):=\int f(\eta)\alpha_\eta(g(\eta\inv
    \gamma))d\lambda^{r(\gamma)}(\eta)\quad\text{and}\quad
    f^*(\gamma)=\alpha_\gamma(f(\gamma\inv)^*). 
  \end{equation*}
\end{prop}

If $f\in\sa_{c}(G,r^{*}\A)$, we define
\begin{equation*}
  \|f\|_{I}=\max\Bigl\{ \sup_{u\in\go} \int
  \|f(\gamma)\|\,d\lambda^{u}(\gamma), \sup_{u\in\go}\int
  \|f^{*}(\gamma)\|\,d\lambda^{u}(\gamma)\Bigr\}, 
\end{equation*}
and say that $\pi:\sa_{c}(G,r^{*}\A)\to B(\H)$ is in $\Rep(G,A)$ if
$\|\pi(f)\|\le \|f\|_{I}$ for all $f\in\sa_{c}(G,r^{*}\A)$.  Then
\begin{equation*}
  \|f\|:=\sup \set{\|\pi(f)\|: \pi\in \Rep(G,A)}
\end{equation*}
defines a norm on $\sa_{c}(G,r^{*}\A)$. The crossed product,
$\A\rtimes_{\alpha}G$, is the completion of $\sa_{c}(G,r^{*}\A)$ with
respect to $\|\cdot\|$.

\begin{remark}
  [Notation] When working with groupoid dynamical systems
  $(\A,G,\alpha)$ it is a
  matter of taste whether to emphasize the bundle $\A$ or the
  \cs-algebra $A=\sa_{0}(\go,\A)$.  Hence many authors write
  $A\rtimes_{\alpha}G$ in place of $\A\rtimes_{\alpha}G$.  We did this
  in the introduction so that our main theorem looks like an
  associative law: $A\rtimes(G\ltimes X\rtimes G)\cong (A\rtimes
  (G\ltimes X))\rtimes H$.  However, in the remainder of this paper we
  will stick with the bundle notation.  This has a number of
  advantages --- for example, see Remark~\ref{rem-g-action} where the
  same \cs-algebra is the section algebra of different bundles.
\end{remark}

We say that a uniformly convergent net in $\sa_{c}(G,r^{*}\A)$ is
eventually compactly supported if there is an $i_{0}$ and a compact
set $K$ such that $\supp f_{i}\subset K$ for all $i\ge i_{0}$.
 There exists a topology on $\sa_c(G,
r^*\A)$ such that $F:\sa_c(G, r^*\A)\to V$ is
continuous into a locally convex linear space $V$ if and only if $F$
maps eventually compactly supported uniformly convergent nets to
convergent nets \cite[Lemma~D.10]{rw:morita}.  We 
call this the \emph{inductive limit topology} on $\sa_c(G, r^*\A)$.
Note that convergence in the inductive limit topology implies
convergence in the $\|\cdot\|_{I}$-norm, and hence in the \cs-norm
$\|\cdot\|$.

\begin{remark}[Separability Assumptions]
  \label{rem-separable-basic}
  In the sequel, we will be exclusively interested in \emph{separable}
  dynamical systems $(\A,G,\alpha)$.  By this we mean that $G$ is
  second countable and that $A:=\sa_{0}(\go,A)$ is separable.  This
  not only allows us to use the Disintegration Theorem and other
  results from \cite{muhwil:nyjm08}, but 
  has a number of other important consequences:
  \begin{enumerate}
  \item The total space, $\A$, is a second countable
    topological space.
  \item The $*$-algebra $\sa_{c}(G,r^{*}\A)$ is separable in the inductive
    limit topology.  (That is, there is a countable dense set $D$
    which is dense in the inductive limit topology.)
  \item If $(\A,G,\alpha)$ is separable, then the crossed product,
    $\A\rtimes_{\alpha}G$ is a separable \cs-algebra.
  \item Any nondegenerate $*$-homomorphism $\pi:\sa_{c}(G,r^{*}\A)\to
    B(\H)$ which is continuous with respect to
the inductive limit topology on $\sa_c(G, r^*\A)$ and the
weak-$*$ topology on $B(\mathcal{H})$ is in $\Rep(G,A)$.  (Since the
converse is automatic, we can view $\Rep(G,A)$ as the set of inductive
limit continuous representations in the separable case.)
  \end{enumerate}
In the case that $\A$ is a (continuous) \cs-bundle, assertion~(a)
follows from \cite{fd:representations1}*{Proposition~II.13.21}.  The
proof in general carries over easily using
\cite{wil:crossed}*{Theorem~C.25} to describe a basis for the topology
on $\A$.  Assertion~(b) is just
\cite{fd:representations1}*{Proposition~II.14.10} in the Banach bundle
case.  The proof in the general case follows \emph{mutatis mutandis}.
Assertion~(c) follows from assertion~(b).
Assertion~(d) is a consequence of the Disintegration Theorem
\cite{muhwil:nyjm08}*{Theorem~7.12}. 
\end{remark}


\section{The Iterated Crossed Product}
\label{sec:iter-cross-prod}

Throughout, $G$ and $H$ will be second countable, locally compact
Hausdorff groupoids with Haar systems $\sset{\lambda^u}_{u\in\go}$ and
$\sset{\sigma^v}_{v\in \ho}$, respectively.  We also fix a
$(G,H)$-equivalence $X$ as in \cite{mrw:jot87}*{Definition~2.1}.  Thus
there are maps $r_{X}:X\to\go$ and $s_{X}:X\to\ho$ and commuting free
and proper actions of $G$ and $H$, respectively.  (We will quickly
drop the subscript `$X$' from $r_{X}$ and $s_{X}$ since the domain
should be clear from context.)  Then we can define the following groupoid.

\begin{definition}\label{def: groupoid product}
  We set
  \[
  \Eu = G\ltimes X\rtimes H := \set{(\gamma,x,\eta)\in G\times X\times H :
  r(\gamma) = r(x)\ \text{and}\ s(x) = r(\eta)},
  \]
  and give $\Eu$ the subspace topology inherited from the product
  topology on $G\times X\times H$.  The groupoid operations are given by
  \begin{equation*}
    (\gamma,x,\eta)(\xi,\gamma\inv \cdot x \cdot \eta,\zeta) = (\gamma\xi, x,
    \eta\zeta)\quad\text{and}\quad(\gamma,x,\eta)\inv = (\gamma\inv,\gamma\inv
    \cdot x\cdot \eta,\eta\inv).
  \end{equation*}
Then we can identify $\eo$ with $X$ so that the range and source maps
are given by
\begin{equation*}
    s(\gamma,x,\eta) = \gamma\inv \cdot x \cdot \eta\quad\text{and} \quad
    r(\gamma,x,\eta) = x.
\end{equation*}
We define a Haar system on $\Eu$ by $\lambda_{\Eu}^x =
  \lambda^{r(x)}\times \delta_x \times \sigma^{s(x)}$ where $\delta_x$
  is the Dirac $\delta$-measure at $x$.
\end{definition}

It is routine to verify that with these operations, $\Eu$ is a second
countable locally compact Hausdorff groupoid with Haar system
$\sset{\lambda_{\Eu}^{x}}_{x\in X}$.

\begin{remark}[The Group Case]
  \label{rem-group-case}
  To see that $\Eu$ is a natural iterated construct in the groupoid
  realm, consider the situation where $\Ggr$ and $\Hgr$ are groups
  acting freely and properly on the left and right, respectively, of
  $X$.  Then we can form the transformation groupoids $G:=\Ggr\ltimes
  X/H$ and $H:=\Ggr\backslash X\rtimes \Hgr$, and let them act on $X$ in
  the natural way so that $X$ becomes a $(G,H)$-equivalence.  Suppose that
  $A=\sa_{0}(X,\A)$ is a $C_{0}(X)$-algebra with commuting $\Ggr$ and
  $\Hgr$ actions $ \alpha$ and $ \beta$, respectively, which induce 
  the given actions on $X$ as in
  \cite{muhwil:nyjm08}*{Example~4.8}.  Then, just as in
  \cite{muhwil:nyjm08}*{Example~4.8}, the crossed product
  $A\rtimes_{\alpha\times\beta}(\Ggr\times\Hgr)$ is isomorphic to the
  groupoid crossed product
  $\A\rtimes_{(\alpha\times\beta)^{\sim}}((\Ggr\times\Hgr)\rtimes 
  X)$ for an appropriate action $(\alpha\times\beta)^{\sim}$.  If
  we let
  $\Eu=(G\ltimes X\rtimes H)$, then the map
  $\bigl((s,x\cdot \Hgr),x,(G\cdot x,t)\bigr) \mapsto (s,t,x)$ is an
  groupoid isomorphism of $\Eu = (\Ggr\ltimes X/\Hgr)\ltimes X \rtimes
  (\Ggr\backslash X\rtimes \Hgr)$ onto $(\Ggr\times \Hgr)\rtimes X$
  which intertwines an action $\omega$ with $(\alpha\times\beta)^{\sim}$.
  Hence the groupoid crossed product $\A\rtimes_{\omega}\Eu$ is
  (isomorphic to) the iterated crossed product
  $A\rtimes_{\alpha\times\beta}(\Ggr \times \Hgr)$.

\end{remark}

  We will identify the transformation groupoid $\Gu:=G\ltimes X/H$ with
  the closed subgroupoid $\set{(\gamma,x,s(x))\in\Eu:r(\gamma)=r(x)}$ of
  $\Eu$.  Thus we will often write $(\gamma,x)\in\Gu$ in place of
  $(\gamma,x,s(x))$.  We equip $\Gu$ with the Haar system
  $\lgb=\sset{\lg^{r(x)}\times\delta_{x}}_{x\in X}$.  Similar
  statements hold for the transformation groupoid $\Hu=X\rtimes H$.
  Notice that $\Eu$, $\Gu$ and $\Hu$ all have unit spaces identified
  with $X$.

If $\Eu$ acts continuously on an upper semicontinuous
\cs-bundle $\A$ by isomorphisms, then the same is true of any closed
subgroupoid.  Hence we get the following proposition.
\begin{prop}
  \label{prop:2}
  Let $(\A,\Eu,\omega)$ be a groupoid dynamical system.  Then
  the restrictions
  \begin{equation*}
    \alpha_{(\gamma,x)}:=\omega_{(\gamma,x,s(x))}\quad\text{and}\quad
    \alpha'_{(x,\eta)} =\omega_{(r(x),x,\eta)}
  \end{equation*}
are continuous actions of $\Gu$ and $\Hu$, respectively, by
isomorphisms on $\A$.
\end{prop}

To obtain the inner portion of our iterated crossed product, we form
the crossed product $B:=\A\rtimes_{\alpha}\Gu$.

\begin{remark}
  \label{rem-g-action}
  We can also view $B$ as a crossed product by $G$.  Since $r_{X}:X\to \go$ is
  continuous, there is an upper semicontinuous \cs-bundle $\A'$ over
  $\go$  such that $A:=\sa_{0}(X,\A) \cong \sa_{0}(\go,\A')$.
  Furthermore there is an 
  induced action $\tilde\alpha$ of $G$ on $\A'$ such that
  $\A'\rtimes_{\tilde \alpha}G$ is isomorphic to $\A\rtimes\Gu$.  (See
  \cite{fmw:pams04}*{Theorem~2} for the details in the case where $A$
  has continuous trace.)
\end{remark}

Since $X$ is a $(G,H)$-equivalence, the source map $s_{X}$ factors
through a homeomorphism of $G\backslash X$ with $\ho$.  In particular, $G\backslash X$ is
Hausdorff and we can identify $v\in\ho$ with the orbit
$s_{X}^{-1}(v)$.  Thus the following proposition follows immediately
from \cite{goe:hjm10}*{Proposition~4.2}.

\begin{prop}
  \label{prop:1}
  Let $H$ and $B$ be as above. Then $B$ is a $C_0(\ho)$-algebra with
  respect to
  the action
  \begin{equation*}
    \phi\cdot f(\gamma,x) = \phi(s(x))f(\gamma,x)
  \end{equation*}
  for $\phi\in C_c(\ho)$ and $f\in
  \sa_{c}(\underline{G},r^*\A)$.  Furthermore, the restriction
  map $\sa_c(\underline{G},r^*\A)\to
  \sa_c(\underline{G}\restr{s\inv(v)},r^*\A)$ factors to an
  isomorphism of the fibre $B(v)$ with $\A\restr{s^{-1}(v)} \rtimes
  \underline{G}\restr{s\inv(v)}$.
\end{prop}

\begin{remark}
  \label{rem-separable}
  Since $B$ is
  separable by Remark~\ref{rem-separable-basic}(c) and since $\ho$ is second
  countable, $\B$ must be second countable as in
  Remark~\ref{rem-separable-basic}(a). 
\end{remark}

We will write $B_{0}(v)$ for the dense $*$-subalgebra
$\sa_{c}(\Gu\restr{s^{-1}(v)},r^{*}\A)$ of $B(v)$. If
  $f\in\sa_{c}(\Gu,r^{*}\A)$ and $v\in\ho$, then we'll write $f_{v}$
  for the element of $B_{0}(v)$ obtained by restriction.  We use the
  set of such sections to define a topology on
  $\B:=\coprod_{v\in\ho}B(v)$ as in \cite{wil:crossed}*{Theorem~C.25}
  making $\B$ an upper semicontinuous \cs-bundle.  Note that $B\cong
  \sa_{0}(\ho,\B)$ and $v\mapsto f_{v}$ is a prototypical section in
  $\sa_{c}(\ho,\B)$.  We can now build the
  outer action of our iterated crossed product.

\begin{prop}
  \label{prop:3}
  Let $(\A,\Eu,\omega)$ be a separable groupoid
  dynamical system with $H$ and $\B$ as above.  Then there is a
  groupoid dynamical system $(\B,H,\beta)$ where, for $f\in
  B_{0}(s(\eta))$,
  \begin{equation}\label{eq:1}
    \beta_{\eta}(f)(\gamma,x)=\omega_{(r(x),x,\eta)}
    \bigl(f(\gamma,x\cdot\eta)\bigr). 
  \end{equation}
\end{prop}
\begin{proof}
  Since multiplication by $\eta$ is a homeomorphism of
  $s^{-1}_{X}(s_{H}(\eta))$ onto $s_{X}^{-1}(r_{H}(\eta))$, it is not
  hard to check that \eqref{eq:1} defines a $*$-homomorphism of
  $B_{0}(s(\eta))$ into $B_{0}(r(\eta))$.  Since $\beta_{\eta}$ is
  $\|\cdot\|_{I}$-isometric, it extends to all of $B(s(\eta))$.  Elementary calculations show that $\beta_{v}=\id$ if $v\in\ho$, and
  that for composable $\eta$ and $\zeta$ we have
  $\beta_{\eta\zeta}=\beta_{\eta}\circ\beta_{\zeta}$. Thus
  $\eta\mapsto \beta_{\eta}$ is an action of $H$ on $\B$ by
  isomorphisms.  It only remains to show that the action is continuous.

 Since $\B$ is second countable by Remark~\ref{rem-separable}, we can
 work with sequences.  Thus we assume that $\eta_i\to \eta_0$ in $H$
 and $b_{i}\to b_{0}$ in $\B$ with $b_{i}\in B(s(\eta_{i}))$ for all
 $i$.  We need to verify that $\beta_{\eta_{i}}(b_{i})\to
 \beta_{\eta}(b)$ in $\B$.

Fix
  $\epsilon > 0$.  Let $u_{i}=s(\eta_{i})$ and $v_{i}=r(\eta_{i})$ for
  all $i\ge0$.   Choose $b\in B:=\A\rtimes_{\alpha}\Gu$ such that $b(u_0) =
  b_0$, and let $F\in \sa_c(\Gu,r^*\A)$ be such
  that $\|F-b\|< \epsilon /2$.  Observe that
  \begin{equation}
    \label{eq:2}
    \|F_{u} - b(u)\| < \epsilon/2\quad\text{for all $u\in \ho$}
  \end{equation}
  where $F_{u}$ denotes the
  restriction of $F$ to $\Gu\restr{s\inv(u)}$.  We first show that it
  suffices to 
  prove the following claim:
  \begin{claim}
    If $F\in \Gamma_c(\Gu ,r^*\A)$, $\eta_i\to\eta_0$
    and $u_i, v_i$ are as above, then $\beta_{\eta_i}(F_{u_i}) \to
    \beta_{\eta_0}(F_{u_0})$ in $\B$.
  \end{claim}

Suppose that claim is valid.  By \eqref{eq:2}
we have
\begin{equation*}
\|\beta_{\eta_0}(F_{u_0}) - \beta_{\eta_0}(b_0)\| = \|F_{u_0} - b(u_0)\|
< \epsilon/2 < \epsilon.
\end{equation*}
Since both $b_i \to b_0$ and $b(u_i)\to b(u_0)=b_0$ we have $\|b(u_i)
- b_i \|\to 0$.  Consequently it follows that for large $i$
\begin{equation*}
\|\beta_{\eta_i}(F_{u_i}) - \beta_{\eta_i}(b_i)\| \leq
\|F_{u_i}-b(u_i)\| + \|b(u_i)-b_i\| < \epsilon.
\end{equation*}
It follows from \cite[Proposition C.20]{wil:crossed} that
$\beta_{\eta_i}(b_i) \to \beta_{\eta_0}(b_0)$ in $\B$ as required.  

Thus it will suffice to prove the claim.  Since it will suffice to see
that every subsequence of $\sset{\beta_{\eta_{i}}(F_{u_{i}})}$ has a
subsequence converging to $\beta_{\eta_{0}}(F_{u_{0}})$, we can
replace $\sset{\beta_{\eta_{i}}(F_{u_{i}})}$ by a subsequence,
relabel, and find a convergent subsequence.  If $v_i = v_0$ infinitely
often, then we 
can pass to another subsequence and assume that $v_i = v_0$ for all $i\geq
0$.  Then, since the relative topology of $B(v_{0})$ in $\B$ is the
norm topology,\footnote{The proof is the same as that for (continuous)
  Banach bundles in
  \cite{fd:representations1}*{Proposition~II.13.11}.}  we can assume
by way of contradiction 
that $\sset{\beta_{\eta_{i}}(F_{u_{i}})}$ does not converge to
$\beta_{\eta_{0}}(F_{u_{0}})$ 
in the inductive limit topology.  Since for any compact neighborhood
  $D$ of $\eta_0$, the supports of $\beta_{\eta_i}(F(u_i))$ are
  eventually 
  contained in the compact set $\{(\gamma,x,\eta):(\gamma,x)\in
  \supp(F), \eta\in D\}$, it follows that
  $\sset{\beta_{\eta_{i}}(F_{u_{i}})}$ does not converge to 
$\beta_{\eta_{0}}(F_{u_{0}})$ uniformly.  It follows that, by passing to
another subsequence and relabeling, we can assume that there exists a
$\delta>0$ such that for each $i>0$ we can pick $(\gamma_i,x_i)$ such
that
  \begin{equation}
    \label{eq:3}
    \|\beta_{\eta_i}(F_{u_i})(\gamma_i,x_i) -
    \beta_{\eta_0}(F_{u_0})(\gamma_i,x_i)\| \geq \delta > 0.
  \end{equation}
  If equation \eqref{eq:3} is to hold we must either have
  $(\gamma_i,x_i\cdot \eta_i)\in \supp (F)$ infinitely often, or
  $(\gamma_i,x_i\cdot \eta_0)\in \supp (F)$ infinitely often.  In either
  case we may pass to a subsequence and multiply by the appropriate
  $H$ elements to find $(\gamma_0,x_0)\in\Gu $ such that
  $(\gamma_i,x_i)\to(\gamma_0,x_0)$.  We then have
  \begin{equation*}
  F(\gamma_i,x_i\cdot \eta_i)\to F(\gamma_0,x_0\cdot \eta_0),\quad\text{and}\quad
  F(\gamma_i,x_i\cdot \eta_0)\to F(\gamma_0,x_0\cdot \eta_0).
  \end{equation*}
  Since $\omega$ is continuous it follows that
  $\beta_{\eta_i}(F_{u_i})(\gamma_i,x_i)$ and
  $\beta_{\eta_0}(F_{u_0})(\gamma_i,x_i)$ both converge to
  $\beta_{\eta_0}(F_{u_0})(\gamma_0,x_0)$. This contradicts
  \eqref{eq:3}, and thus the claim holds in this case.

  On the other hand, suppose that we may remove an initial segment and
  assume that $v_i \ne v_0$ for all $i > 0$.  Then we may also pass to
  a subsequence and assume that $v_i \ne v_j$ for all $i \ne j$.  Let
  $\Omega = \{v_i\}_{i=0}^\infty$ and define $\iota$ on
  $s_X\inv(\Omega)$ by $\iota(x) = i$ if and only if $s_X(x) = v_i$.
  Note that $\Omega$ is compact.  It is straightforward to show that
  $x\mapsto \eta_{\iota(x)}$ is continuous on $s_X\inv(\Omega)$. Since
  $\omega$ and $F$ are also continuous,
  \begin{equation*}
    F_0(\gamma,x) := \beta_{\eta_{\iota(x)}}(F_{u_{\iota(x)}})(\gamma,x) = 
    \omega_{(r(x),x,\eta_{\iota(x)})}(F(\gamma,x\cdot \eta_{\iota(x)}))
  \end{equation*}
  defines an element of
  $\sa_c(\Gu _{s_X\inv(\Omega)},r^*\A)$.  By the Tietze Extension
  Theorem for upper semicontinuous Banach bundles
  \cite{muhwil:dm08}*{Proposition~A.5}, we can assume that
  $F_{0}\in\sa_{c}(\Gu,r^{*}\A)$.    Because $v\mapsto (F_0)_{v}$ is a
  continuous section of $\B$, 
  \begin{equation*}
 \beta_{\eta_i}(F_{u_i}) =(F_0)_{v_i}  \to (F_0)_{v_0} =
  \beta_{\eta_0}(F_{u_0})
  \end{equation*}
  with respect to the topology on $\B$.  This completes the proof.
\end{proof}


\section{The Main Theorem}
\label{sec:isom-iter-prod}

Now that we have constructed the action of $H$ on $\A\rtimes \Gu$ we
may state the main result of the paper. 

\begin{thm}
  \label{thm:isomorphism}
  Let $G$ and $H$ be second countable locally compact Hausdorff
  groupoids and $X$ a $(G,H)$-equivalence.  Suppose that
  $(\A,\Eu,\omega)$ is a separable groupoid dynamical
  system.  Let $\alpha$ be the restriction of $\omega$ to
  $\Gu = G\ltimes X$ and $B = \A\rtimes_\alpha \Gu =\sa_{0}(H,\B)$.
  Let $\beta$ be the action of $H$ on $\B$ given by
  \begin{equation*}
    \beta_\eta(f)(\gamma,x) = \omega_{(r(x),x,\eta)}(f(\gamma,x\eta)).
  \end{equation*}
  Then there is an isomorphism $\Upsilon : \A\rtimes_{\omega} \Eu \to
  \B\rtimes_{\beta} H$ characterized by 
  \begin{equation*}
    \Upsilon(f)(\eta)(\gamma,x) := f(\gamma,x,\eta)
  \end{equation*}
 for all $f\in \Gamma_c(\Eu,r^*\A)$.
\end{thm}

The proof of Theorem~\ref{thm:isomorphism} is  involved, and we divide
it up into a series of propositions.  We begin by showing that
$\Upsilon$ is a $*$-homomorphism.  Although this assertion is
considerably easier than the assertion that the map is injective, even
this part of the result is technical and far from immediate.  Showing that
$\Upsilon$ is isometric is subtle and requires 
delicacies with unitary groupoid representations that we 
address in a separate section.

\begin{prop}
  \label{prop:4}
  The map $\Upsilon$ defined in Theorem~\ref{thm:isomorphism} extends
  to a $*$-homomorphism from $\A\rtimes_{\omega} E$ onto
  $\B\rtimes_{\beta} H$.
\end{prop}
\begin{proof}
  If $f\in\sa_{c}(\Eu,r^{*}\A)$, then 
  $\Upsilon(f)(\eta)\in B_{0}(r(\eta))$.  Furthermore $\Upsilon(f)$ is a
  compactly supported section of $r^{*}\B$.  Thus to see that
  $\Upsilon$ maps into $\B\rtimes_{\beta}H$ we need to show that
  $\eta\mapsto \Upsilon(f)(\eta)$ is continuous.  The
  topology on $\B$ is determined by the sections coming from
  $\sa_{c}(\Gu,r^{*}\A)$, and is second countable
  (Remark~\ref{rem-separable}).  Hence we can proceed as in
  Proposition~\ref{prop:3}.  We just sketch the details.

  Suppose $\eta_i \to \eta_0$ in $H$.  Replacing
  $\sset{\Upsilon(f)(\eta_{i})}$ by a subsequence, it suffices to see
  that it has a subsequence converging to $\Upsilon(f)(\eta_{0})$.  If
  $r(\eta_i) = r(\eta_0)$ infinitely often, then after passing to a
  subsequence, we may assume $\Upsilon(f)(\eta_i)\in B(r(\eta_0))$ for
  all $i$.  Since we are dealing with a fixed fibre it suffices to
  show that $\|\Upsilon(f)(\eta_n)-\Upsilon(f)(\eta_0)\|\to
    0$ and this follows from a
  standard argument.  Alternatively, if we eventually have
  $r(\eta_i)\ne r(\eta_0)$ then we may pass to a subsequence and
  assume $r(\eta_i) \ne r(\eta_j)$ for all $i\ne j$.  Next, as in the
  proof of Proposition \ref{prop:3}, we build a continuous function
  $F_0\in\sa_{c}(\Gu,r^{*}\A)$ with the property that the restriction
  of $F_0$ to $s\inv(r(\eta_i))$ is equal to $\Upsilon(f)(\eta_i)$ for
  all $i \geq 0$.  Since $F_0$ defines a continuous section
    $v\mapsto F_0|_{s\inv(v)}$, we have
    $\Upsilon(f)(\eta_i)=F_0|_{s\inv(r(\eta_i))}\to
    F_0|_{s\inv(r(\eta_0))}=\Upsilon(f)(\eta_0)$.   In either
  case $\Upsilon(f)$ is a continuous, compactly supported section.

  The next step is to see that $\Upsilon$ is a $*$-homomorphism on
  $\sa_{c}(\Eu,r^{*}\A)$.  This is mostly routine, but some care is
  required to see that it is multiplicative:
  $\Upsilon(f*g)=\Upsilon(f)*\Upsilon(g)$.  The issue is that \emph{a
    priori} $\Upsilon(f)*\Upsilon(g)(\eta)$ is the element of
  $B(r(\eta)) = \A\restr{s^{-1}(r(\eta))} \rtimes_{\alpha}
  \Gu\restr{s^{-1}(r(\eta))}$ given by the \emph{$B(r(\eta))$-valued}
  integral
  \begin{equation*}
    \int_{H}\Upsilon(f)(\zeta)*\beta_{\zeta}\bigl(\Upsilon(g)(\zeta^{-1}\eta)\bigr)
    \,d\lh^{r(\eta)}(\zeta). 
  \end{equation*}
  We claim that first, $\Upsilon(f)*\Upsilon(g)(\eta)$ belongs to
  $B_{0}(r(\eta))$, and second that
  \begin{multline}
    \label{eq:6}
    \int_{H}\Upsilon(f)(\zeta)*\beta_{\zeta}\bigl(\Upsilon(g)(\zeta^{-1}\eta)\bigr)
    \,d\lh^{r(\eta)}(\zeta)(\gamma,x) =\\ \int_{H}
    \Upsilon(f)(\zeta)*\beta_{\zeta}\bigl(\Upsilon(g)
    (\zeta^{-1}\eta)\bigr)(\gamma,x) 
    \,d\lh^{r(\eta)}(\zeta).
  \end{multline}
  This will suffice since a routine computation shows that the
  $A(x)$-valued integral on the right-hand side of \eqref{eq:6}
  simplifies to
  \begin{equation*}
    f*g(\gamma,x,\eta)=\Upsilon(f*g)(\eta)(\gamma,x).
  \end{equation*}
  However both claims follow just as in
  \cite{wil:crossed}*{Lemma~1.108}, and $\Upsilon$ is a
  $*$-homomorphism as claimed.

  To show that $\Upsilon$ is bounded with respect to the \cs-norms, we
  note that it suffices to prove that $\Upsilon$ is continuous with
  respect to the inductive limit topologies.  Then if $L$ is a
  faithful representation of $\B\rtimes_{\beta}H$, $L\circ \Upsilon$
  is a bounded representation of $\A\rtimes_{\omega}\Eu$ by
  \cite{muhwil:nyjm08}*{Theorem~7.12}, and we have
  \begin{equation*}
    \|f\|\le \|L\circ\Upsilon(f)\|=\|\Upsilon(f)\|.
  \end{equation*}

  To see that $\Upsilon$ is continuous in the inductive limit
  topologies, suppose that $f_i \to f_0$ with respect to the inductive
  limit topology in $\Gamma_c(E,r^*\A)$ and let $K$ be the compact set
  which eventually contains the supports of the $f_i$.  Pick
  $\epsilon>0$.
  Let $K_{G}$ be
  the restriction of $K$ to $G$ and observe that $K_{G}$ is compact so
  that the set $\{\lambda^u(K_{G}),\lambda_u(K_{G})\}_{u\in \go}$ is
  bounded by some $M$.  Pick $i_0$ so that $\|f_i-f_0\|_\infty <
  \epsilon/M$ for all $i> i_0$.  Then given any $\eta\in H$ and $x\in
  X$ such that $s(x) = r(\eta)$ we have
  \begin{align*}
    \int \|f_i(\gamma,x,\eta) - f_0(\gamma,x,\eta)\|
    d\lambda^{r(x)}(\gamma)
    \leq \|f_i-f_0\|_\infty \lambda^{r(x)}(K_{G}) < \epsilon, \quad\text{and}\\
    \int \|f_i(\gamma,x,\eta) - f_0(\gamma,x,\eta)\|
    d\lambda_{r(x)}(\gamma)
    \leq \|f_i-f_0\|_\infty \lambda_{r(x)}(K_{G}) < \epsilon. \\
  \end{align*}
  It follows that for all $i > i_0$ we have
  \begin{equation*}
    \|\Upsilon(f_i)(\eta)-\Upsilon(f_0)(\eta)\| \leq
    \|\Upsilon(f_i)(\eta)-\Upsilon(f_0)(\eta)\|_I < \epsilon.
  \end{equation*}
  Consequently, $\Upsilon(f_i)\to\Upsilon(f_0)$ uniformly and, since
  the supports of the $\Upsilon(f_i)$ are eventually contained in the
  restriction of $K$ to $H$, this convergence occurs with respect to
  the inductive limit topology.

  It only remains to see that $\Upsilon$ is surjective.  But
  $\set{\Upsilon(f)(\eta):f\in\sa_{c}(\Eu,r^{}\A)}$ is dense in
  $B(r(\eta))$ and it follows from
  \cite{wil:crossed}*{Proposition~C.24} that
  $\Upsilon(\sa_{c}(\Eu,r^{*}\A))$ is dense in $\sa_{c}(H,r^{*}\B)$ in
  the supremum norm.  But then a compactness argument implies the
  image of $\Upsilon$ is dense in the inductive limit topology.  Hence
  $\Upsilon$ has dense image and is necessarily onto.
\end{proof}

Thus to prove Theorem~\ref{thm:isomorphism}, we just need to see that
$\Upsilon$ is isometric.  The idea of the proof is elementary: if $R$
is a representation of $\A\rtimes_{\omega}\Eu$, then we want to show
that it \emph{factors through} $\B\rtimes_{\beta}H$.  More precisely,
we will prove the following.
\begin{prop}
  \label{prop-finish}
  If $R$ is a representation of $\A\rtimes_{\omega}\Eu$, then there is
  a representation $L$ of $\B\rtimes_{\beta}H$ such that $L\circ
  \Upsilon$ is equivalent to $R$.
\end{prop}

Of course, once we've proved Proposition~\ref{prop-finish}, it follows
easily that $\Upsilon$ is isometric: If $R$ is a faithful representation
of $\A\rtimes_{\omega}\Eu$, then
\begin{equation*}
  \|f\|=\|R(f)\|=\|L\circ \Upsilon(f)\|\le \|\Upsilon(f)\|.
\end{equation*}
Since Proposition~\ref{prop:4} implies $\|\Upsilon(f)\|\le\|f\|$, this
completes the proof of Theorem~\ref{thm:isomorphism}.  

However our
proof of Proposition~\ref{prop-finish} requires that we work with
covariant representations of groupoid dynamical systems, and unitary
groupoid representations in particular.  While the corresponding
details in the group case are straightforward, working out the
niceties for groupoids is extremely subtle.  We do this in the next
sections after briefly reviewing the necessary definitions.


\section{Proof of Proposition~\ref{prop-finish}}
\label{sec:proof-that-upsilon}

We refer to \cite{wil:crossed}*{Appendix~F} for background on Borel
Hilbert bundles and \cite{muhwil:nyjm08}*{\S7} for background on
covariant representations of groupoid dynamical systems.  For
convenience, we review some of the basic concepts here.

\subsection{Covariant Representations}
\label{sec:covar-repr}

Let $\HH=\sset{\H(y)}_{y\in Y}$ be a collection of Hilbert spaces
indexed by an analytic Borel space $Y$.  The disjoint union $Y*\HH$,
viewed as a bundle $p:Y*\HH\to Y$ in the obvious way, is called a
\emph{Borel Hilbert bundle} if it has a natural Borel structure
respecting the Hilbert space structure on the $\H(y)$.  (For a
precise statement, see \cite[Definition~F.1]{wil:crossed}.)  By
\cite[Proposition~F.6]{wil:crossed} there exists a sequence of
sections $e_i:Y\to Y*\HH$ called \emph{a special orthogonal
  fundamental sequence} such that $\set{e_i(x): e_i(x)\neq 0}$ is an
orthonormal basis for $\H(x)$ and such that $h:Y\to Y*\HH$ is Borel if
and only if the maps $x\mapsto \bip(h(x)|{ e_i(x)})_{\H(x)}$ are Borel for
all $i$.  If $\mu$ is a measure on $X$ then $(h,k)\mapsto \int_Y
\bip(h(x)|{k(x)})_{\H(x)} d\mu(x)$ defines an inner product on the bounded
Borel sections of $Y*\HH$.  We denote the Hilbert space completion of
these sections by $L^2(Y*\HH,\mu)$. Given a Borel Hilbert bundle
$Y*\HH$, its isomorphism groupoid $\Iso(Y*\HH)$ is the set
$\{(x,V,y): \text{$V$ a unitary from $H(y)$ to $H(x)$}\}$ with
multiplication $(x,V, y)(y,U, z)=(x,VU, z)$ and the weakest Borel
structure such that $(x,V,y)\mapsto \bip(Vh(y)|{k(x)})$ is Borel for all bounded Borel
sections $h$ and $k$.

Let $S$ be a locally compact Hausdorff groupoid with Haar system
$\{\kappa^u\}_{u\in S^{(0)}}$.  Let $\mu$ be a Radon measure on
$S^{(0)}$ and define $\nu= \mu \circ \kappa := \int_{S^{(0)}} \kappa^u d\mu(u)$.  We say $\mu$ is \emph{quasi-invariant} if $\nu$ is
equivalent to its image under inversion
\cite[Definition~3.2]{ren:groupoid}. We define the \emph{modular
  function} $\Delta:=\Delta^{S}_{\mu}$ to be the Radon-Nikodym
derivative of $\nu$ with respect to its image under inversion.  We
will use the fact that the modular function can be taken to be
multiplicative (see \cite{muhwil:nyjm08}*{Remark~7.1}). If $\mu$ is an
arbitrary Radon measure on $S^{(0)}$ and $\nu_0$ is a finite measure
on $S$ equivalent to $\nu=\mu\circ\kappa$, then we define the
saturation of $\mu$ to be the push-forward $[\mu]=s_*\nu_0$.\footnote{Note
  that the saturation depends on our choice of $\nu_{0}$ and so is
  well-defined only up to equivalence of measures.}  It is shown in
\cite[Proposition~3.6]{ren:groupoid} that the saturation $[\mu]$ is
quasi-invariant, and if $\mu$ is quasi-invariant to begin with, then
$\mu$ is equivalent to $[\mu]$.

Let $(\CC,S,\vartheta)$ be a groupoid dynamical system.  Following
\cite[Definition~7.9]{muhwil:nyjm08}, a covariant representation
$(\pi, U, S^{(0)}*\HH, \mu)$ of $(C,S,\vartheta)$ consists of a Borel
Hilbert bundle $S^{(0)}*\HH$ over $S^{(0)}$, a quasi invariant measure
$\mu$, a Borel field of representations $\pi_u: C(u)\to H(u)$ and a
Borel homomorphism $U:S\to \Iso(S\unit*\HH)$ that satisfies the
\emph{covariance condition}: there is a $\nu$-null set $N$
such that for all $\gamma\notin N$
\begin{equation}\label{eq:covariant}
  U_\gamma
  \pi_{s(\gamma)}(b)=\pi_{r(\gamma)}(\vartheta_\gamma(b))U_\gamma\quad\text{for
    all $b\in C(s(\gamma))$.} 
\end{equation}
It will be convenient to recall that if
$(\pi, U, S^{(0)}*\HH, \mu)$ is covariant, then there is a
$\mu$-conull set $V\subset S\unit$ such that the covariance condition
\eqref{eq:covariant} holds for all $\gamma\in S\restr V$
\cite{muhwil:nyjm08}*{Remark~7.10}.  (Notice that $S\restr V$ is
$\nu$-conull.)

By \cite[Proposition~7.11]{muhwil:nyjm08}, each covariant
representation $(\pi, U, S^{(0)}*\HH, \mu)$ determines a
representation $\pi\rtimes U$ of $\CC\rtimes_{\vartheta}S$ on
$L^2(S^{(0)}*\HH,\mu)$ (called \emph{the integrated form}) such that
\begin{equation*}
  \pi\rtimes U(f)h(u)=\int_G \pi_u(f(\gamma))U_\gamma
  h(s(\gamma))\Delta^{S}_{\mu}(\gamma)^{-1/2} d\kappa^u(\gamma) 
\end{equation*}
for $f\in\sa_{c}(S,r^{*}\CC)$ and $h\in L^2(S^{(0)}*\HH,\mu)$.  It suffices to define $U$ only on a restriction
of the form $S\restr V$ where $V$ is a $\mu$-conull set in $S\unit$.
By
\cite[Theorem~7.12]{muhwil:nyjm08} every representation of
$\CC\rtimes_\vartheta S$ is equivalent to the integrated form of a
covariant representation.

\subsection{A \boldmath Representation of $B$}
\label{sec:quasi-invariance}

Let $(\pi,U,X*\HH,\mu)$ be a covariant representation of
$(A,\Eu,\omega)$.  The first step will be to build a covariant
representation of $(\A,\Gu,\alpha)$.  We already have the Borel
Hilbert bundle $X*\HH$ and $\pi$ is already a Borel field of
representations of $\A$.  The restriction of $U$ from $\Eu$ to $\Gu $,
which we shall denote $U_{\Gu}$, is still a Borel homomorphism into
$\Iso(X*\HH)$.  Furthermore, since $U_{\Gu}$ and $\alpha$ are the
restrictions of $U$ and $\omega$ to $\Gu $, respectively, there is a
$\mu$-conull set $V\subset X$ such that the covariance relation
between $U_{\Gu}$ and $\pi$ holds on $\Gu\restr V$.  Thus it will follow
that $(\pi,U_{\Gu}, X*\HH,\mu)$ is a covariant representation of
$(\A,\Gu,\alpha)$ provided we can show that $\mu$ is quasi-invariant
with respect to $\Gu$.

\begin{prop}
  \label{prop:5}
  Let $\mu$ be a quasi-invariant measure on $X$ with respect to $\Eu$.
  \begin{enumerate}
  \item\label{quasi trans} Then $\mu$ is quasi-invariant with respect
    to $\Gu $ and $\Hu$.
  \item \label{quasi push} The push forward measure $\tau=s_*\mu$ on
    $H\unit$ is quasi-invariant with respect to~$H$.
  \end{enumerate}
\end{prop}

\begin{proof}
  \eqref{quasi trans} Since $\mu$ is quasi-invariant on $X$ with
  respect to $\Eu$, it is equivalent to its saturation $[\mu]$ with
  respect to $\Eu$ \cite[Proposition~3.6]{ren:groupoid}.  So it will
  suffice to show that $[\mu]$ is quasi-invariant with respect to $\Gu
  $.  Let $\nu=\mu\circ\lambda_{\Eu}$ be the measure induced on
  $\Eu$ by $\mu$, $\nu_0$ a finite measure equivalent to $\nu$, and
  $d\nu/d\nu_0$ the (strictly positive) Radon-Nikodym derivative of
  $\nu$ and $\nu_0$.  By definition, $[\mu] = s_*\nu_0$.  To see
  that $[\mu]$ is quasi-invariant with respect to $\Gu$, it will
  suffice to see that if $f\in C^+_c(\Gu )$ is such that $\iint
  f(\gamma,x) \,d\lambda^{r(x)}(\gamma)\,d[\mu](x)=0$, then $\iint
  f(\xi\inv,\xi\inv \cdot x) \,d\lambda^{r(x)}(\xi)\,d[\mu](x)=0$.  Now,
  \begin{align*}
    0 &= \iint f(\gamma,x) \,d\lambda^{r(x)}(\gamma)\,d[\mu](x) \\
    &= \iint f(\gamma,\xi\inv \cdot x \cdot \eta)
    d\lambda^{s(\xi)}(\gamma)\,d\nu_0(\xi,x,\eta) \\
    &= \iiiint f(\gamma,\xi\inv \cdot x\cdot
    \eta)\frac{d\nu}{d\nu_0}(\xi,x,\eta)
    \,d\lambda^{s(\xi)}(\gamma)
    \,d\lambda^{r(x)}(\xi)\,d\sigma^{s(x)}(\eta)\,d\mu(x) 
    \\ 
    &= \iiiint f(\xi\inv\gamma,\xi\inv \cdot x\cdot \eta)
    \frac{d\nu}{d\nu_0}(\xi,x,\eta)
    \,d\lambda^{r(x)}(\gamma)
    \,d\lambda^{r(x)}(\xi)\,d\sigma^{s(x)}(\eta)\,d\mu(x) \\ 
    &= \iiiint f(\xi\inv, \xi\inv\gamma\inv \cdot x\cdot \eta)
    \frac{d\nu}{d\nu_0}(\gamma\xi,x,\eta)
    \,d\lambda^{s(\gamma)}(\xi)\,d\lambda^{r(x)}(\gamma)
    \,d\sigma^{s(x)}(\eta)\,d\mu(x) \\ 
    &=\iint f(\xi\inv, \xi\inv\gamma\inv \cdot x\cdot \eta)
    \frac{\frac{d\nu}{d\nu_0}(\gamma\xi,x,\eta)}
    {\frac{d\nu}{d\nu_0}(\gamma,x,\eta)} \,d\lambda^{s(\gamma)}(\xi)
    \,d\nu_0(\gamma ,x,\eta).
  \end{align*}
  So off a $\nu_0$-null set we have
  \begin{equation}
    \label{eq:9}
    0 = \int f(\xi\inv,\xi\inv\gamma\inv \cdot x\cdot \eta)
    \frac{\frac{d\nu}{d\nu_0}(\gamma\xi,x,\eta)}
    {\frac{d\nu}{d\nu_0}(\gamma,x,\eta)} 
    \,d\lambda^{s(\gamma)}(\xi).
  \end{equation}
  However, because $d\nu/d\nu_0$ is strictly positive, the supports of
  \begin{equation*}
    \xi\mapsto f(\xi\inv,\xi\inv\gamma\inv \cdot x\cdot\eta)
    \frac{\frac{d\nu}{d\nu_0}(\gamma\xi,x,\eta)}
    {\frac{d\nu}{d\nu_0}(\gamma,x,\eta)} 
    \quad \text{and} \quad
    \xi\mapsto f(\xi\inv,\xi\inv\gamma\inv \cdot x \cdot\eta)
  \end{equation*}
  are the same. Thus, since $f\geq 0$, \eqref{eq:9} holds if and only
  if
  \begin{equation*}
    0 = \int f(\xi\inv,\xi\inv \gamma\inv \cdot x\cdot  \eta)
    d\lambda^{s(\gamma)}(\xi).
  \end{equation*}
  As a result
  \begin{align*}
    0&= \iint f(\xi,\xi\inv\gamma\inv \cdot x\cdot\eta)
    \,d\lambda^{s(\gamma)}(\xi)
    \,d\nu_0(\gamma,x,\eta) \\
    &= \iint f(\xi\inv,\xi\inv \cdot x)\,d\lambda^{r(x)}(\xi)\,d[\mu](x).
  \end{align*}
  It follows that $[\mu]$ is quasi-invariant with respect to $\Gu
  $. The corresponding assertion for $\Hu$ follows by symmetry.

  The proof of \eqref{quasi push} is similar but
  easier.
\end{proof}

Now that we have a covariant representation $(\pi,U_{\Gu},X*\HH,\mu)$
of $(\A, \Gu , \alpha)$ we may form the integrated representation $R =
\pi\rtimes U_G$ of $B$ on $L^2(X*\HH,\mu)$.  This will make up the
$C^*$-algebraic portion of a covariant representation of $(\B,
H,\beta)$.
\subsection{A \boldmath Representation of $H$}
\label{sec:build-repr-h}

  Next we must build a unitary representation of $H$.  We
use the fact that $\mu$ is quasi-invariant with respect to $\Hu$
(Proposition~\ref{prop:5}) to restrict $U$ to a unitary representation
$(U_{\Hu},X*\HH,\mu)$ of the groupoid $\Hu$.  This yields a
representation of the transformation groupoid $C^*$-algebra
$C^*(\Hu)$.  However, it is routine to check that $C^*(\Hu)$ is
naturally isomorphic to the groupoid crossed product
$C_0(X)\rtimes_{\rt} H$ (for example, see \cite[Remark
2.7]{goe:imj09}).  Next, we recall from the proof of the
Disintegration Theorem for crossed products \cite[Theorem
7.12]{muhwil:nyjm08}, that there is a nondegenerate map from $C^*(H)$ into the
multiplier algebra $M(C_0(X)\rtimes H)$ defined for $\phi \in C_c(H)$
and $f\in \Gamma_c(H,r^*C_0(X))$ by 
\[
\phi \cdot f(\eta) = \int \phi(\zeta) \rt_\zeta(f(\zeta\inv\eta))
\,d\sigma^{r(\eta)}(\zeta).
\]
Since the latter is
isomorphic to $M(C^*(\Hu))$, we obtain a map $m$ of $C^{*}(H)$ into
$M(C^{*}(\Hu))$ which is given on $\phi\in C_c(H)$ and $f\in C_c(\Hu)$ by
\begin{equation}
  \label{eq:11}
  m(\phi)f(x,\eta) = \int \phi(\zeta)f(x\cdot \zeta,\zeta\inv\eta)
  d\sigma^{r(\eta)}(\zeta).
\end{equation}
This multiplier action is important because we can use it to form a
representation $W$ of $H$ on $L^2(X*\HH,\mu)$.

\begin{prop}
  \label{prop:6}
  Let $m$ be the action given in \eqref{eq:11} and let
  $\overline{U}_{\Hu}$ be the extension of (the integrated form) of
  $U_{\Hu}$ to the multipliers of $C^{*}(\Hu)$.  Then $W =
  \overline{U}_{\Hu}\circ m$ is a representation of $C^*(H)$ which,
  for $\phi\in C_c(H)$ and $h\in L^2(X*\HH,\mu)$, is given by 
  \begin{equation}
    \label{eq:12}
    W(\phi)h(x) = \int \phi(\eta)
    U_{(r(x),x,\eta)}h(x\cdot \eta)\Delta^{\Hu}_{\mu}\neghalf(x,\eta)
    d\sigma^{s(x)}(\eta). 
  \end{equation}
\end{prop}

\begin{proof}
  Since $\overline{U}_{\Hu}$ is a representation of $M(C^*(\Hu))$ by
  \cite{ren:jot87}*{Proposition~3.5(ii)} and since
  $m$ is a nondegenerate $*$-homomorphisms, $W$ is a
  representation.  That $W$ has the form given in
  \eqref{eq:12} follows from the following computation.  By
  nondegeneracy it suffices to verify \eqref{eq:12} for vectors of the
  form $U_{\Hu}(f)h$ for $f\in C_c(\Hu)$ and $h\in L^2(X*\HH,\mu)$.
  Let $\phi\in C_c(H)$. Then
  \begin{align*}
    W(\phi)&U_{\Hu}(f)h(x) = U_{\Hu}(m(\phi)f)h(x) \\
    &=\int m(\phi)f(x,\zeta)U_{(r(x),x,\zeta)}h(x\cdot
    \zeta)\Delta^{\Hu}_\mu\neghalf(x,\zeta)
    \,d\sigma^{s(x)}(\zeta) \\
    &=\iint \phi(\eta)f(x\cdot \eta,\eta\inv\zeta)
    U_{(r(x),x,\zeta)}h(x\cdot \zeta)\Delta^{\Hu}_\mu\neghalf(x,\zeta)
    \,d\sigma^{s(x)}(\eta)\,d\sigma^{s(x)}(\zeta)
    \\
    &=\iint \phi(\eta)f(x\cdot
    \eta,\zeta)U_{(r(x),x,\eta\zeta)}h(x\cdot
    \eta\zeta)\Delta^{\Hu}_\mu\neghalf(x,\eta\zeta) \,d\sigma^{s(\eta)}(\zeta)
    \,d\sigma^{s(x)}(\eta)
    \\
    &=\int \phi(\eta)U_{(r(x),x,\eta)} \int f(x\cdot
    \eta,\zeta)U_{(r(x),x\cdot\eta,\zeta)} h(x\cdot \eta\zeta)
    \Delta^{\Hu}_\mu\neghalf(x\cdot\eta,\zeta) \,d\sigma^{s(\eta)}(\zeta) \\
    &\hspace{3in} \Delta^{\Hu}_\mu\neghalf(x,\eta) \,d\sigma^{s(x)}(\eta)\\
    &=\int \phi(\eta)U_{(r(x),x,\eta)}U_{\Hu}(f)h(x\cdot
    \eta)\Delta^{\Hu}_\mu\neghalf(x,\eta)\,d\sigma^{s(x)}(\eta).  \qedhere
  \end{align*}
\end{proof}

In order to proceed any further we need to examine the measure $\mu$
and the relationship between the various modular functions more
closely.  But first we need to recall some issues concerning
disintegration of measures from \cite{wil:crossed}*{Appendices F~and
  I}.  Let $\mu$ be a measure on $X$ and $\tau = s_*\mu$ be the push
forward of $\mu$ to $H\unit$.  We now use the Disintegration Theorem
for measures \cite[Theorem I.5]{wil:crossed} to generate Radon
measures $\{\mu_u\}_{u\in H\unit}$ on $X$ such that:
\begin{enumerate}
\item Off a $\tau$-null set $N$, each $\mu_u$ is a probability measure
  supported on $s\inv(u)\subset X$ and $\mu_u = 0$ for $u\in N$.
\item For all bounded Borel functions $h$, the map
  \begin{equation*}
    u \mapsto \int h(x)d\mu_u(x)
  \end{equation*}
  is bounded and Borel.
\item For all bounded Borel functions $h$,
  \begin{equation*}
    \label{eq:13}
    \int_X h(x) d\mu(x) = \iint h(x) d\mu_u(x) d\tau(u).
  \end{equation*}
\end{enumerate}

Let $\sset{e_i}$ be the special orthogonal fundamental sequence for a
Borel Hilbert bundle $X*\HH$ over $X$. By \cite[Example
F.19]{wil:crossed} there exists a Borel Hilbert bundle $H^{(0)}*\KK$
with fibres $\K(u) = L^2(s\inv(u)*\HH,\mu_u)$ for all $u$ and so that
\begin{equation*}
  \tilde{e}_i(u)(x) = e_i(x)
\end{equation*}
defines a special orthogonal fundamental sequence for $H\unit*\KK$.
Furthermore, the map $V:L^2(X*\HH,\mu)\to L^2(H\unit*\KK,\tau)$ given
by $V(h)(u)(x) = h(x)$ is a natural isomorphism between the two
spaces.

By Proposition \ref{prop:5}, $\tau$ is quasi-invariant with respect to
$H$.  Thus we may form the modular function on $H$,
$\Delta^{H}_{\tau}(\eta)$, as well as the usual one,
$\Delta^{\Hu}_{\mu}(x,\eta)$, on $\Hu$.  Our next proposition connects the
decomposition of $\mu$ with respect to $\tau$ with the action of $H$
on $X$.  (As a special case, it simply says that $H$ and $\tau$ give
us a measured groupoid as in \cite{muh:cbms}*{Chap.~4}.)

\begin{prop}
  \label{prop:7}
  Let $\{\mu_u\}_{u\in H\unit}$ be the decomposition of $\mu$ as
  above.  There is a $\tau$-conull set $V\subset\ho$ such that for all
  $\eta\in H\restr V$ and all bounded Borel functions $\phi$ we have
  \begin{equation}
    \label{eq:14}
    \int \phi(x) d\mu_{r(\eta)}(x) = \int \phi(x\cdot\eta) \theta(x,\eta)
    d\mu_{s(\eta)}(x)
  \end{equation}
  where $\theta(x,\eta) = \Delta^{\Hu}_\mu(x,\eta)/\Delta^H_\tau(\eta)$.
\end{prop}

Proving Proposition~\ref{prop:7} will take some work.  We start by
using the groupoid structure to generate a very special ``invariant''
section of $X*\HH$.
\begin{lemma}
  \label{lem:2}
  There exists $e\in \B(X*\HH)$ such that $e(x)$ is a unit vector if $\H(x)\ne 0$ and for all
  $(x,\eta)\in\underline{H}$
  \begin{equation*}
    e(x\cdot \eta) = U^*_{(r(x),x,\eta)}e(x).
  \end{equation*}
\end{lemma}

\begin{proof}
  Since second countable, locally compact Hausdorff spaces are  Polish
  --- see \cite{wil:crossed}*{Lemma~6.5} --- and 
  since the orbit space $X/H$ is Hausdorff, 
  it follows from the Corollary
  following \cite{arv:invitation}*{3.4.1} that there is a Borel cross
  section $c$ for the quotient map $q:X\to X/H$. Using $c$ we may
  construct 
  a Borel map $\varsigma:X\to H$ with the property that
  \begin{equation*}
    c(q(x))\varsigma(x) = x\quad\text{for all $x\in X$.}
  \end{equation*}

  Now let $e_i$ be a special orthogonal fundamental sequence for
  $X*\HH$ as in \cite[Remark F.7]{wil:crossed}.  By definition,
  $e_1(x)$ is a unit vector whenever $\H(x)$ is non-trivial.  We
  define
  \begin{equation*} 
    e(x) = U_{(r(x),c(q(x)),\varsigma(x))}^* e_1(c(q(x))).  
  \end{equation*}
  Since $U$ is a unitary representation, $e(x)$ is a unit vector if
  $\H(x)\ne 0$.  All that remains is to show that 
  $e$ is Borel and invariant.  According to the definition of a
  fundamental sequence, to show that $e$ is Borel it suffices to show
  that $x \mapsto \bip(e(x)|{e_i(x)})$ is Borel for all $i$
  (\cite{wil:crossed}*{Proposition~F.6}).  For this note that
  \begin{equation*}
    (\xi,x,\eta)\mapsto \bip(U^{*}_{(\xi,x,\eta)}e_{1}(x) |
    {e_{i}(\xi^{-1}\cdot x \cdot \eta)})
  \end{equation*}
  is Borel on $\Eu$ (since $U$ is a representation).  Therefore, since $x\mapsto
  (r(x),c(q(x)),\varsigma(x))$ is Borel, so is
  \begin{equation*}
    x\mapsto \bip(e(x)|{e_{i}(x)}) =
    \bip(U^{*}_{(r(x),c(q(x)),\varsigma(x)))} e_{1}(c(q(x)))) | {e_{i}(x)}).
  \end{equation*}
  The ``invariance'' portion of the lemma now follows from a brief
  computation and the observation that
  $\varsigma(x\cdot\eta)=\varsigma(x)\eta$.
\end{proof}

This special invariant vector, combined with the representation $W$,
is the key to showing that the measure decomposition respects the
groupoid action.

\begin{proof}[Proof of Proposition \ref{prop:7}]
  Let $h,k\in\L^2(X*\HH,\mu)$ and $\psi\in C_c(H)$.  Using
  \eqref{eq:12},
we have
   that
  \begin{align*}
    (W(\psi)&h|k) = \int \bip(W(\psi)h(x)|{k(x)})\,d\mu(x) \\
    &=\iiint \psi(\eta)\bip(U_{(r(x),x,\eta)}h(x\cdot\eta)|{k(x)})
    \Delta^{\Hu}_{\mu}\neghalf(x,\eta)d\sigma^{s(x)}(\eta)\,d\mu_u(x)\,d\tau(u)
    \\
    &= \iint \psi(\eta) \int \bip(U_{(r(x),x,\eta)}h(x\cdot
    \eta)|{k(x)})\Delta^{\Hu}_\mu\neghalf(x,\eta) d\mu_{r(\eta)}(x) d\sigma^u(\eta) \,d\tau(u).
  \end{align*}
  Given $\eta\in H$ define the Radon measure $\eta\cdot \mu_{s(\eta)}$
  supported in $r^{-1}_{X}(r(\eta))$ by
  \begin{equation*}
    \int \varrho(x\cdot \eta)d(\eta\cdot \mu_{s(\eta)})(x) := \int
    \varrho(x)d\mu_{s(\eta)}(x)\quad \text{for all $\varrho\in C_c(X)$.}
  \end{equation*}
  Using the fact that $W$ is a representation, we may also write
  \begin{align*}
    \bip(&W(\psi)h|k) = \overline{\bip(W(\psi^*)k|h)}  \\
    &=\iiint
    \overline{\psi^*(\eta)\bip(U_{(r(x),x,\eta)}k(x\cdot\eta)|{h(x)})}
    \Delta^{\Hu}_\mu\neghalf(x,\eta)
    \,d\sigma^{s(x)}(\eta) \,d\mu_u(x)\,d\tau(u) \\
    &=\iiint \psi(\eta\inv)\bip(h(x)|{U_{(r(x),x,\eta)}k(x\cdot
      \eta)})\Delta^{\Hu}_\mu\neghalf(x,\eta)
    d\mu_{r(\eta)}(x) \,d\sigma^u(\eta) \,d\tau(u) \\
    &=\iiint \psi(\eta)\bip(h(x)|{U_{(r(x),x,\eta\inv)}k(x\cdot
      \eta\inv)})\Delta^{\Hu}_\mu\neghalf(x,\eta\inv)
    \Delta^H_\tau(\eta)\inv\,d\mu_{s(\eta)}(x) \\
    &\hspace{3in} d\sigma^u(\eta) \,d\tau(u)  \\
    &=\iint \psi(\eta) \int \bip(U_{(r(x),x,\eta)}h(x\cdot  
  \eta)|{k(x)})\Delta^{\Hu}_\mu\neghalf(x,\eta)\theta(x,\eta) \,d(\eta\cdot
    \mu_{s(\eta)})(x)\\
    &\hspace{3in} d\sigma^u(\eta) \,d\tau(u),
  \end{align*}
  where we used the groupoid homomorphism properties of $U$ and
  $\Delta^{\Hu}_\mu$ to get the last equality.  Since these two forms of
  $\bip(W(\psi)h|k)$ are equal for all $\psi\in C_c(H)$, we may
  conclude that given  $h$ and $k$ there is a
  $\tau\circ\lh$-null set $N_{h,k}$ such that $\eta\not\in N_{h,k}$
  implies
  \begin{align}
    \label{eq:15}
    \int &\bip(U_{(r(x),x,\eta)}h(x\cdot \eta)|{k(x)})\Delta^{\Hu}_\mu\neghalf
    (x,\eta)
    \,d\mu_{r(\eta)}(x) \\
    \nonumber &=\int \bip(U_{(r(x),x,\eta)}h(x\cdot
    \eta)|{k(x)})\Delta^{\Hu}_\mu\neghalf (x,\eta)
    \theta(x,\eta)\,d(\eta\cdot\mu_{s(\eta)})(x).
  \end{align}
(To be precise, we are choosing $N_{h,k}$ so that both sides of
\eqref{eq:15} are well-defined, finite and equal to each other.)

  Let $e$ be the invariant vector from Lemma \ref{lem:2}.  Then, by
  construction, we have for all $\varrho\in C_c(X)$
  \begin{equation}
    \label{eq:16}
    \bip(U_{(r(x),x,\eta)}e(x\eta)|{\varrho e(x)}) =
    \bip(e(x)|{\varrho(x)e(x)}) = \varrho(x) 
  \end{equation}
  for all $x$ such that $\H(x)$ is nontrivial.  Let
  $\{\varrho_i\}$ be an inductive limit dense set in $C_c(X)$
  (see Remark~\ref{rem-separable-basic}(b)).    If we set $h = e$ and $k =
  \varrho_i e$, then we can conclude from \eqref{eq:15} and
  \eqref{eq:16} that for all $\eta\not\in N_{e,\varrho_i e}$
  \begin{equation}
    \label{eq:17}
    \int \varrho_i(x)\Delta^{\Hu}_\mu\neghalf(x,\eta) \,d\mu_{r(\eta)}(x) = 
    \int \varrho_i(x)\Delta^{\Hu}_\mu\neghalf(x,\eta)\theta(x,\eta) \,d(\eta\cdot
    \mu_{s(\eta)})(x). 
  \end{equation}
  We may take the countable union $N = \bigcup_i N_{e,\varrho_i e}$ to
  obtain a null set $N$ such that \eqref{eq:17} holds for all $i$
  provided $\eta\not\in N$.  We can  assume
  that if $K\subset X$ is compact, then there is an $i_{0}$ such that
  $\varrho_{i_{0}}\ge0$ and equal to $1$ on $K$. This
  implies that the measures $\Delta^{\Hu}_\mu(x,\eta)^{-\half}\,d\mu_{r(\eta)}$
  and $\Delta^{\Hu}_\mu(x,\eta)^{-\half}\theta(x,\eta)\,d(\eta\cdot
  \mu_{s(\eta)})$ are finite on compact subsets of $X$.  Thus they are
  Radon measures by \cite{rud:real}*{Theorem~2.18}.  In particular,
  they are determined on $C_{c}(X)$.  Since the $\varrho_i$ are dense in the
  inductive limit topology, it follows that the two measures are
  equal.  Thus
  \eqref{eq:17} holds for
  all nonnegative Borel functions.  Replacing an arbitrary nonnegative
  Borel
  function $\phi(x)$ by $\phi(x)\Delta\poshalf(x,\eta)$ we conclude
  that for $\eta\not\in N$ equation \eqref{eq:14} holds for all
  nonnegative Borel functions.  Since the $\mu_{u}$ are probability
  measures, \eqref{eq:14} holds for 
all bounded Borel 
  functions as claimed.

  It is clear that the set $\Sigma = \set{\eta\in H:
    \text{\eqref{eq:14} holds}}$ is conull.  Since the modular
  functions are all homomorphisms it is straightforward to show that
  if $\eta,\zeta\in\Sigma$ such that $s(\eta)=r(\zeta)$ then
  $\eta\zeta\in\Sigma$. It follows from a result of Ramsay's (see
  \cite{ram:am71}*{Lemma~5.2} or \cite{muh:cbms}*{Lemma~4.9}) that
  there is a $\tau$-conull set $V$ such that $H\restr V \subset
  \Sigma$.  This completes the proof.
\end{proof}

The proof of the following lemma is a brief computation and has been
omitted.

\begin{lemma}
  \label{lem:3}
  Let $\nu= \mu\circ \lambda_{\Eu}$.
   Then, $\nu$-almost everywhere, we have
  \begin{equation*}
    \Delta^{\Gu}_\mu(\gamma,x)\Delta^{\Hu}_\mu(\gamma\inv \cdot x, \eta) 
    = \Delta^{\Eu}_\mu(\gamma,x,\eta)
    =\Delta^{\Gu}_\mu(\gamma,x\cdot \eta)\Delta^{\Hu}_\mu(x,\eta).
  \end{equation*}
\end{lemma}

\subsection{Back to the Proof of Proposition~\ref{prop-finish}}
\label{sec:back-proof-prop}

Now that we have dealt with the major measure theoretic issues, we can
turn to Proposition~\ref{prop-finish}.

\begin{proof}[Proof of Proposition~\ref{prop-finish}] Let
  $R=\pi\rtimes U_{\Gu}$ be the representation of $B$ given by the
  integrated form of $(\pi, U_{\Gu}, X*\HH,\mu)$.  We construct a groupoid
  representation of $H$ which is covariant with $R$.  From the
  discussion before the statement of Proposition~\ref{prop:7}, we can
  use \cite[Example~F.19]{wil:crossed} to obtain a Borel Hilbert
  bundle $H^{(0)}*\KK$ from $X*\HH$ with fibres $\K(u) =
  L^2(s\inv(u)*\HH,\mu_u)$ so that $V:L^2(X*\HH,\mu)\to
  L^2(H\unit*\KK,\tau)$ given by $V(h)(u)(x) = h(x)$ is a natural
  isomorphism between the two spaces. Note that if $\{e_i\}$ is a
  special orthogonal fundamental sequence for $X*\HH$, then
  $\tilde{e}_i(u)(x)=e_i(x)$ is a special orthogonal fundamental
  sequence for $H\unit*\KK$.

  Define a representation $Q$ of $B$ on $L^2(H\unit*\KK,\tau)$ by $Q =
  VRV^*$.  Simple computations show that $Q$ is a $C_0(H\unit)$-linear
  homomorphism and that for $f\in B_0(u)$ and $x\in s_X\inv(u)$
  \begin{equation*}
    (Q_u(f)h)(u)(x) = \int
    \pi_x(f(\gamma,x))U_{(\gamma,x,u)}h(u)(\gamma\inv\cdot
    x)\Delta^{\Gu}_\mu\neghalf(\gamma,x) \,d\lambda^{r(x)}(\gamma).
  \end{equation*}
  Next we define $W_\eta : \K(s(\eta))\to \K(r(\eta))$ by
  \begin{equation*}
    (W_\eta h)(r(\eta))(x) =
    U_{(r(x),x,\eta)}h(s(\eta))(x\cdot \eta)\theta\neghalf(x,\eta).
  \end{equation*}
  Straightforward computations using Proposition~\ref{prop:7} show
  that there is a $\tau$-conull set $V\subset\ho$ such that for all
  $\eta\in H\restr V$, $W_\eta$ is a unitary and that $W_\eta W_\zeta
  = W_{\eta\zeta}$ when $\eta$ and $\zeta$ are composable in $H\restr
  V$.  The
  last thing we need to check is that $W$ is Borel.  Given
  $\tilde{e}_j$ and $\tilde{e}_k$ we have
  \begin{equation*}
    \bip(W_\eta \tilde{e}_j(s(\eta))|{ \tilde{e}_k(r(\eta))}) = \int
    \bip(U_{(r(x),x,\eta)} e_j(x\cdot\eta)|{e_k(x)}) \theta\neghalf(x,\eta)
    \,d\mu_{r(\eta)}(x).
  \end{equation*}
  The integrand is Borel with respect to $\eta$ because the $e_i$ are
  a fundamental sequence and $U$ is a Borel representation.  Since
  $\mu_u$ is a Borel field of measures it follows from standard,
  albeit lengthy, arguments that
  \begin{equation*}
    \eta\mapsto \bip(W_\eta\tilde{e}_j(s(\eta))|{\tilde{e}_k(r(\eta))})
  \end{equation*}
  is Borel.  Thus we have all of the components to form a Borel
  representation $(W,H\unit*\KK,\tau)$ of $H$.

\begin{remark}
  Although we won't make use of this fact, the integrated form of
  $W_\eta$ is the push forward of the representation $W$ defined in
  Proposition~\ref{prop:6} from $L^2(X*\HH)$ to $L^2(H\unit*\KK)$ by
  $V$.
\end{remark}

We now show that $(Q, W, H\unit*\KK,\tau)$ satisfies the covariance
condition.  Using Lemma~\ref{lem:3},
\begin{align*}
  &(W_\eta Q_{s(\eta)}(f)h)(r(\eta))(x) =
  U_{(r(x),x,\eta)}Q_{s(\eta)}(f)h(x\cdot\eta)\theta\neghalf(x,\eta) \\
  &= \int U_{(r(x),x,\eta)}\pi_{x\cdot\eta}(f(\gamma,x\cdot \eta))
  U_{(\gamma,x\cdot\eta,s(\eta))}h(\gamma\inv \cdot x \cdot
  \eta)\\
  &\hspace{6 cm}\Delta^{\Gu}_\mu\neghalf(\gamma,x\cdot\eta)\theta\neghalf(x,\eta)
  \,d\lambda^{r(x)}(\gamma)  \\
  &= \int
  \pi_x(\omega_{(r(x),x,\eta)}(f(\gamma,x\cdot\eta)))U_{(\gamma,x,\eta)}
  h(\gamma\inv \cdot x\cdot \eta)\\
  &\hspace{6 cm}
  \Bigl(\frac{\Delta^{\Gu}_\mu(\gamma,x\cdot\eta)
    \Delta^{\Hu}_\mu(x,\eta)}{\Delta^H_\tau(\eta)}\Bigr)^{-\half} \,d\lambda^{r(x)}(\gamma) \\
  &= \int
  \pi_x(\beta_\eta(f)(\gamma,x))U_{(\gamma,x,\eta)}h(\gamma\inv \cdot
  x \cdot \eta) \Bigl(\frac{\Delta^{\Gu}_\mu(\gamma,x)\Delta^{\Hu}_\mu(\gamma\inv
    \cdot x,\eta)}{\Delta^H_\tau(\eta)}\Bigr)^{-\half} \,d\lambda^{r(x)}(\gamma) \\
  &= \int \pi_x(\beta_\eta(f)(\gamma,x))U_{(\gamma,x,s(x))}(W_\eta
  h)(r(\eta))(\gamma\inv \cdot x)
  \Delta^{\Gu}_\mu\neghalf(\gamma,x) \,d\lambda^{r(x)}(\gamma) \\
  &= Q_{r(\eta)}(\beta_\eta(f))W_\eta h(r(\eta))(x).
\end{align*}

This shows that $(Q,W,H\unit*\KK,\tau)$ is a covariant
representation of $(\B,H,\beta)$.  It remains to show that $V$
intertwines $\pi\rtimes U$ and $(Q\rtimes W)\circ \Upsilon$.  Given
$h,k\in L^2(X*\HH)$ and $f\in \Gamma_c(E,r^*\A)$,
\begin{align*}
  &\bip(V^*Q\rtimes W(\Upsilon(f))Vh|k) = \bip(Q\rtimes W(\Upsilon(f))Vh|Vk) \\
  &= \iiint \bip(Q(\Upsilon(f)(\eta))W_\eta
  Vh(s(\eta))(x)|{Vk(r(\eta))(x)})
  \Delta^H_\tau\neghalf(\eta) \,d\mu_u(x)\,d\sigma^u(\eta)\,d\tau(u) \\
  &=\iiint \bip(\pi_x(f(\gamma,x,\eta))U_{(\gamma,x,s(x))}W_\eta
  Vh(s(\eta))(\gamma\inv
  \cdot x)|{k(x)})\Delta^H_\tau\neghalf(\eta)\Delta^{\Gu}_\mu\neghalf(\gamma,x) \\
  &\hspace{6 cm}
  d\lambda^{r(x)}(\gamma) \,d\sigma^{s(x)}(\eta) \,d\mu(x) \\
  &=\iiint \bip(\pi_x(f(\gamma,x,\eta))U_{(\gamma,x,\eta)}
  h(\gamma\inv \cdot x\cdot \eta)|{k(x)}) \theta\neghalf(\gamma\inv
  \cdot x,\eta)\Delta^{\Gu}_\mu\neghalf(\gamma,x) \\
  &\hspace{6 cm} \Delta^H_\tau\neghalf(\eta) d\lambda^{r(x)}(\gamma) \,d\sigma^{s(x)}(\eta)
  \,d\mu(x) \\
  &=\iiint \bip(\pi_x(f(\gamma,x,\eta))U_{(\gamma,x,\eta)}
  h(\gamma\inv \cdot x \cdot
  \eta)|{k(x)}) \Delta^{\Eu}_{\mu}\neghalf(\gamma,x,\eta)\\
  &\hspace{6 cm} \,d\lambda^{r(x)}(\gamma)
  \,d\sigma^{s(x)}(\eta) \,d\mu(x) \\
  &= \bip(\pi\rtimes U(f)h|k).
\end{align*}
Since this holds on a dense subset it holds everywhere and we get the
result.
\end{proof}


\def\noopsort#1{}\def\cprime{$'$} \def\sp{^}

\end{document}